\documentclass{article}

\usepackage{macro}

\usepackage[most]{tcolorbox}

\AtEveryBibitem{%
    \clearname{editor}%
    \clearname{editora}%
    \clearname{editorb}%
    \clearname{editorc}%
}

\hypersetup{
    colorlinks=true,
    linkcolor=purple,
    citecolor=forestgreen,
}

\title{Sharp One-\!Dimensional Sub-Gaussian Comparison in Convex Order}
\author{
    Yihan Zhang\thanks{
        School of Mathematics, University of Bristol. 
        Email: \href{mailto:yihan.zhang@bristol.ac.uk}{\texttt{yihan.zhang@bristol.ac.uk}}.
    }
}

\begin{document}
\maketitle

\begin{abstract}
We prove that any random variable $X$ whose moment generating function is pointwise upper bounded by that of $ G \sim \mathcal{N}(0,1) $ must be dominated by $ G/\mathbb{E}[|G|] $ in convex order, meaning $ \mathbb{E}[f(X)] \le \mathbb{E}[f(G/\mathbb{E}[|G|])] $ for all convex $f$. This is sharp as witnessed by $ X \sim \mathrm{Unif}(\{-1,1\}) $ and $ f(x) = |x| $. 
\end{abstract}

% \clearpage
% \pagenumbering{roman}
% \newpage
% \tableofcontents
% \newpage

% \pagenumbering{arabic}

\newcounter{asmpctr} % counter for assumptions
\setcounter{asmpctr}{\value{enumi}}

%%%%%%%%%%%%%%%%%%%%%%%%%%%%%%
%%%%%%%%%%%%%%%%%%%%%%%%%%%%%%
%%%%%%%%%%%%%%%%%%%%%%%%%%%%%%

\section{Introduction}

It is recently established in \cite[Theorem 1.1]{vanHandel} that there exists a universal constant $c\in(0,\infty)$ such that any centered sub-Gaussian random vector $X$ in $ \bbR^d $ is dominated in convex order by the $c$ multiple of a standard Gaussian vector. 
This motivates the question of characterizing the minimal value of $c$. 
The precise notion of $1$-sub-Gaussianity adopted by \cite{vanHandel} requires 
\begin{align}
&&
    \expt{X} &= 0_d ; & 
    \prob{ \abs{\inprod{v}{X}} > t } &\le \min\brace{ 1, 2 \cdot e^{-t^2/2} } \textnormal{ for all $ t \ge 0 $ and $ v\in\bbS^{d-1} $.} &
& \label{eqn:notion1} 
\end{align}
Under this notion, \cite{Davis_Power} settles the optimal $c$ for the case $d=1$. 

It is well known that \Cref{eqn:notion1} is equivalent, \emph{up to a universal constant}, to 
\begin{align}
    \expt{e^{\lambda \inprod{v}{X}}} &\le e^{\lambda^2/2} \textnormal{ for all $ \lambda \in \bbR $ and $ v\in\bbS^{d-1} $;} \label{eqn:notion2}
\end{align}
see e.g.\ \cite[Definition 3.4.1 and Proposition 2.6.1]{Vershynin}. 
Specifically, there exist universal constants $ c_1,c_2>0 $ such that for any $X$ satisfying \Cref{eqn:notion1}, $ c_1 X $ satisfies \Cref{eqn:notion2}; and for any $X$ satisfying \Cref{eqn:notion2}, $ c_2 X $ satisfies \Cref{eqn:notion1}. 
Therefore the convex dominance result in \cite{vanHandel} also holds under the notion \Cref{eqn:notion2}. 
However, since the conversion between \Cref{eqn:notion1,eqn:notion2} is lossy by constants, it is unclear whether the optimal value of $c$ differs under these two notions. 
In particular, it is at least equally interesting and natural to characterize the analogous optimal constant under \Cref{eqn:notion2}. 

Focusing on the case $d=1$, this paper finds the optimal $c$ under \Cref{eqn:notion2} (see \Cref{thm:cmgf}) which turns out to be strictly smaller and arguably more interpretable than its counterpart identified by \cite{Davis_Power} under \Cref{eqn:notion1} (see \Cref{rk:compare}). 
Our result answers a question of Power \cite{Power}. 

The tractability of the one-dimensional case is unlocked by the hinge representation of convex functions which also plays a central role in \cite{Davis_Power}: any convex function $ f\colon\bbR\to\bbR $ can be written as 
\begin{align}
     f(x) &= ax + b + \int \max\{x-t,0\} \, \mu(\dd t) \label{eqn:hinge} 
\end{align} 
for some $a,b\in\bbR$ and a nonnegative Borel measure $\mu$ on $\bbR$; see \cite[Proposition 3.A.4]{Shaked_Shanthikumar}. 
Indeed, Taylor's theorem with integral remainder identifies $ (a,b,\mu) $ as $ (f'(0), f(0), f'') $ where $ f'' $ is interpreted in the sense of Schwartz distributions. 
Building on this, our analysis proceeds via a key reduction that allows us to focus only on two-point distributions; see \Cref{lem:Y}. 
This aligns with the behavior of the extremal distribution and function saturating the convex dominance relation which we identify to be $ X \sim \unif(\{-1,1\}) $ and $f(x) = \abs{x} = - x + 2 \max\{x,0\}$ corresponding to taking $ (a,b,\mu) = (-1,0,2\delta_0) $ in \Cref{eqn:hinge}. 
See \Cref{eqn:emerge} for more details. 
It then remains to study convex dominance of two-point distributions. 
The main technical ingredient we establish here to conclude the result is a lower bound on the Gaussian isoperimetric function in terms of an elementary function which may be of independent interest; see \Cref{lem:I}. 

In the rest of the paper, we work solely with real-valued random variables. 
All $\log$ and $\exp$ are to the base $e$. 

\section{Main results}

A random variable $X$ is said to be \emph{$1$-sub-Gaussian (in the MGF sense)} if 
\begin{align}
    M_X(\lambda) \le e^{\lambda^2/2} \label{eqn:subG}
\end{align}
for all real $\lambda$, where $ M_X(\lambda) \coloneqq \expt{ e^{\lambda X} } $. 
Denote by $ \cG $ the set of all such random variables. 
For two random variables $ X $ and $ Y $, we write $ X \lecx Y $ if $X$ is dominated by $Y$ in convex order, meaning $ \expt{f(X)} \le \expt{f(Y)} $ for all convex functions $ f\colon\bbR\to\bbR $ such that both expectations exist. 
For $ G \sim \cN(0,1) $, define
\begin{align}
    \cmgf &\coloneqq \inf\brace{ c > 0 : \forall X\in\cG ,\, X \lecx c\,G } . \label{eqn:defc} 
\end{align}

The main deliverable of this paper is the theorem below proved in \Cref{sec:pf}. 

\begin{theorem}
\label{thm:cmgf}
$ \cmgf = \sqrt{\pi/2} $. 
\end{theorem}

\begin{remark}
\label{rk:compare}
There exist other notions of sub-Gaussianity in the literature. 
For instance, a random variable $ X $ is said to be \emph{$1$-sub-Gaussian in the tail sense} if $ \expt{X} = 0 $ and 
\begin{align}
    \prob{\abs{X} > t} \le \min\brace{1 , 2 \cdot e^{-t^2/2}} , \notag 
\end{align}
for all $ t\ge0 $. 
One can then define $ \ctail $ in a way similar to \Cref{eqn:defc} where $ \cG $ is replaced with the collection of all random variables that are $1$-sub-Gaussian in the tail sense. 
The value of $ \ctail $ is identified in \cite[Theorem 1]{Davis_Power} which differs from $ \cmgf $ identified in \Cref{thm:cmgf}. 
We refer to \cite[Equations (3) to (6)]{Davis_Power} for the precise definition of $ \ctail $. 
Numerically, $ \ctail \approx 2.30952 $; see \cite[Remark 2]{Davis_Power}. 
In our case, $ \sqrt{\pi/2} \approx 1.25331 $ emerges as $ 1/\expt{\abs{G}} $; see \Cref{eqn:emerge}.  
\end{remark}

\begin{remark}
Note that since $ x \mapsto e^{\lambda x} $ is convex for any $\lambda\in\bbR$ and $ M_G(\lambda) = e^{\lambda^2/2} $, any $ X \lecx G $ must satisfy \Cref{eqn:subG}. 
\Cref{thm:cmgf} can be viewed as a converse to this, up to a necessary constant $ \sqrt{\pi/2} $. 
\end{remark}

More generally, a random vector $ X \in \bbR^d $ (where $ d\in\bbZ_{\ge1} $) is called \emph{$1$-sub-Gaussian (in the MGF sense)} if \Cref{eqn:notion2} holds. 
Denote by $ \cG(d) $ the set of all such random vectors. 
The notion of convex domination can be defined analogously against all convex functions $ f\colon\bbR^d\to\bbR $. 
For $ G \sim \cN(0_d,I_d) $, let 
\begin{align}
    \cmgf(d) &\coloneqq \inf\brace{ c > 0 : \forall X\in\cG(d) ,\, X \lecx c\,G } . \notag 
\end{align}
Our next result shows that the sharp constant for $d=1$ identified in \Cref{thm:cmgf} is not necessarily sharp for all $d\ge2$. 
This follows by exhibiting a lower bound on $ \cmgf(d) $ that can be strictly larger than $ \cmgf \equiv \cmgf(1) $ for some $ d\ge2 $; see \Cref{app:highd} for a proof. 
Therefore, $ \cmgf(d) $ is genuinely dimension dependent. 

% \begin{theorem}
% \label{thm:2d}
% $ \cmgf(2) \ge \frac{8}{9} \sqrt{\pi\log(2)} $. 
% \end{theorem}

% \begin{remark}
% \Cref{thm:2d} is proved by identifying an explicit sub-Gaussian random vector $ X\in\bbR^2 $ (see \Cref{eqn:X}) and a convex function $ f\colon\bbR^2\to\bbR $ (see \Cref{eqn:f}). 
% In fact, the constant $ \frac{8}{9} \sqrt{\pi\log(2)} $ is sharp for this particular construction of $X$ (see \Cref{lem:X}). 
% However, we are unable to prove a matching upper bound on $ \cmgf(2) $ or find another example that improves upon the lower bound in \Cref{thm:2d}. 
% Characterizing $ \cmgf(2) $ (and indeed $ \cmgf(d) $ for every $ d\in\bbZ_{\ge2} $) is therefore left open. 
% \end{remark}

For $ d\in\bbZ_{\ge2} $, define 
\begin{align}
&&
    \sigma_d &\coloneqq \sqrt{\frac{d^2 - 1}{2d\log(d)}} , & 
    m_{d} &\coloneqq \expt{\max_{i\in[d]} Z_i} , & 
& \label{eqn:sigmad}
\end{align}
where $ Z_1, \cdots, Z_d \iid \cN(0,1) $. 

\begin{theorem}
\label{thm:highd}
For any $ d\in\bbZ_{\ge2} $, $ \cmgf(d) \ge \frac{d}{\sqrt{d+1} \, \sigma_d \, m_{d+1}} $. 
\end{theorem}

\Cref{thm:highd} is proved by identifying an explicit sub-Gaussian random vector $ X\in\bbR^d $ and a convex function $ f\colon\bbR^d\to\bbR $ which are natural high-dimensional generalizations of $ X \sim \unif(\{-1,1\}) $ and $ f(x) = \abs{x} $ leading to a sharp lower bound for $d=1$ in \Cref{thm:cmgf}. 
Specifically, the proof of \Cref{thm:highd} constitutes analyzing $ X \in \bbR^d $ uniformly distributed on the vertices of a regular simplex (suitably scaled to meet the $ 1 $-sub-Gaussian condition) and a piecewise linear function $ f \colon \bbR^d\to\bbR $ supported on the cone partition of $ \bbR^d $ induced by the vertices of the simplex.
See \Cref{eqn:fd,eqn:Xd} for details of the constructions of $X$ and $ f $, respectively. 

\begin{remark}[$d=2$]
\label{rk:2d}
For $d=2$, the lower bound in \Cref{thm:highd} evaluates to $ \frac{8}{9} \sqrt{\pi\log(2)} \approx 1.31170 $ which is strictly larger than $ \cmgf \equiv \cmgf(1) $. 
In fact, the constant $ \frac{8}{9} \sqrt{\pi\log(2)} $ is sharp for this particular construction of $X$ (see \Cref{lem:X} in \Cref{app:sharp2}). 
However, we are unable to prove a matching upper bound on $ \cmgf(2) $ or find another example that improves upon this lower bound. 
Characterizing $ \cmgf(2) $ 
% (and indeed $ \cmgf(d) $ for every $ d\in\bbZ_{\ge2} $) 
is therefore left open. 
\end{remark}

\begin{remark}[General $d$]
\label{rk:highd}
\cite[Theorem 1.1]{vanHandel} implies that $ \cmgf(d) $ is bounded above uniformly in $d\ge1$. 
Moreover, it is not hard to see that $ \cmgf(d) $ is nondecreasing in $d$; see \Cref{lem:mono} in \Cref{app:rk}. 
Therefore, $ \lim_{d\to\infty} \cmgf(d) $ exists and is finite. 
In this regard, the lower bound in \Cref{thm:highd}, denoted by $ \ul{c}(d) $, has an undesirable feature that it does not respect the aforementioned monotonicity property. 
One way to see this is by computing the large-$d$ limit of $ \ul{c}(d) $:   
\begin{align}
    \ul{c}(d)
    &= \frac{d}{d+1} \cdot \sqrt{\frac{d}{d-1}} \cdot \frac{\sqrt{2\log(d)}}{m_{d+1}} 
    \underset{d\to\infty}{\to} 1 < \ul{c}(2) , \notag 
\end{align}
where we have used the well-known asymptotics for $ m_{d+1} $; see e.g.\ \cite[Theorem 5.1]{DasGupta_Lahiri_Stoyanov}.
Therefore as $ d\ge2 $ increases, $ \ul{c}(d) $ must strictly decrease at some point. 
Indeed, numerical evaluation reveals that the first drop occurs from $d=3$ to $d=4$: $ \ul{c}(3) \approx 1.32273 $, $ \ul{c}(4) \approx 1.32262 $. 
This shows that the lower bound $ \ul{c}(d) $ is not sharp for sufficiently large $d$. 
\end{remark}

\paragraph{Preliminaries.}

Denote by 
\begin{align}
&&
    \varphi(x) &= \frac{e^{-x^2/2}}{\sqrt{2\pi}} , & 
    \Phi(x) &= \int_{-\infty}^x \varphi(y) \diff y
& \notag 
\end{align}
the p.d.f.\ and c.d.f.\ of standard Gaussians, respectively. 
Let $ I(p) = \varphi(\Phi^{-1}(p)) $ be the Gaussian isoperimetric function. 
It will also be convenient to work with 
\begin{align}
&&
    \erf(x) &= \frac{2}{\sqrt{\pi}} \int_0^x e^{-y^2} \diff y , & 
    \erfc(x) &= 1 - \erf(x) . & 
& \notag 
\end{align}
One has the relation $ \erfc(x) = 2 (1 - \Phi(\sqrt{2}\,x)) $. 

En route to \Cref{thm:cmgf}, we establish a curious lower bound on $I$ which may be of independent interest. 

\begin{lemma}
\label{lem:I}
For any $ p\in[0,1] $, 
\begin{align}
    I(p) &\ge \frac{2}{\sqrt{\pi}} \cdot p(1-p) \sqrt{\frac{\log\frac{1-p}{p}}{1-2p}} . \label{eqn:I} 
\end{align}
\end{lemma}

The proof of \Cref{lem:I} is technical and is deferred to \Cref{app:ineq}. 

\begin{remark}
\label{rk:I}
In fact, as a consequence of the sharpness of $ \cmgf $ in \Cref{thm:cmgf}, the constant $ 2/\sqrt{\pi} $ on the RHS of \Cref{eqn:I} is also sharp in that it cannot be replaced with a larger constant. 
Numerical evaluation suggests that the lower bound in \Cref{eqn:I} is never lower than $I$ by more than $ 0.00720 $ uniformly on $[0,1]$ and the maximal gap occurs at $ p_0 \approx 0.10125 $ and $ 1 - p_0 \approx 0.89875 $. 
\end{remark}

\section{Proof of \texorpdfstring{\Cref{thm:cmgf}}{}}
\label{sec:pf}

We first quote a well-known result (see \cite[Theorem 3.A.1]{Shaked_Shanthikumar} or \cite[Proposition 4]{Davis_Power}) which asserts that it suffices to check only hinge functions for convex domination. 

For any random variable $ Z $ and any $t\in\bbR$, define $ S_Z(t) \coloneqq \expt{(Z - t)_+} $, where $ (\cdot)_+ \coloneqq \max\brace{\cdot,0} $. 

\begin{proposition}
\label{prop:hinge}
Let $ X,Y $ be integrable random variables. 
Then $ X \lecx Y $ if and only if $ \expt{X} = \expt{Y} $ and $ S_X(t) \le S_Y(t) $ for all $ t\in\bbR $. 
\end{proposition}

Next, we note that any sub-Gaussian random variable must be centered. 

\begin{lemma}
\label{lem:center}
For any $ X\in\cG $, one has $ \expt{X} = 0 $. 
\end{lemma}

\begin{proof}
This is well known; see e.g.\ \cite[Exercise 2.23]{Vershynin}. 
We provide a short verification for completeness. 
Note that for any $ \lambda>0 $, both $ x \mapsto e^{\lambda x} $ and $ x \mapsto e^{-\lambda x} $ are convex. 
Thus Jensen's inequality gives
\begin{align}
&&
    e^{\lambda \expt{X}} &\le \expt{e^{\lambda X}} = M_X(\lambda) , & 
    e^{- \lambda \expt{X}} &\le \expt{e^{- \lambda X}} = M_X(-\lambda) . & 
& \notag 
\end{align}
Since $ X\in\cG $, the RHS's of both inequalities above are upper bounded by $ e^{\lambda^2/2} $, implying 
\begin{align}
    \max\brace{-\expt{X}, \expt{X}} &\le \lambda/2 . \notag 
\end{align}
Taking $ \lambda\downarrow0 $ shows $ \expt{X} = 0 $. 
\end{proof}

By \Cref{prop:hinge} and \Cref{lem:center}, our goal becomes proving 
\begin{align}
    S_X(t) \le S_{\sqrt{\pi/2}\,G}(t) \label{eqn:TODO}
\end{align}
for any $ X\in\cG $ and any $ t\in\bbR $. 

For an arbitrary $ t\in\bbR $, let $ \cE \coloneqq \brace{X>t} $ and $ p \coloneqq \prob{\cE} $. 
If $ p = 0 $, then $ X \le t $ almost surely and hence $ S_X(t) = 0 $. 
In this case, the desired result \Cref{eqn:TODO} is trivial. 
If $ p = 1 $, then $ X > t $ almost surely, and hence $ S_X(t) = \expt{X - t} = -t $. 
On the other hand, for any $ c\in\bbR $, 
\begin{align}
    S_{c\,G}(t) &= \expt{ (c\,G - t)_+ }
    \ge \expt{c\,G - t} = -t = S_X(t) , \notag 
\end{align}
as required by \Cref{eqn:TODO}. 

From now on, we assume $ p \in (0,1) $ so that both $ \cE $ and $ \cE^c $ have positive probabilities. 
Define 
\begin{align}
&&
    \mu_+ &\coloneqq \expt{ X \mid \cE } \ge t , & 
    \mu_- &\coloneqq \expt{X \mid \cE^c} \le t . & 
& \notag 
\end{align}
Let $ Y $ have the law 
\begin{align}
    p \delta_{\mu_+} + (1 - p) \delta_{\mu_-} . \label{eqn:Y} 
\end{align}

The following lemma makes the key reduction which allows us to focus only on two-point distributions. 

\begin{lemma}
\label{lem:Y}
Let $ X\in\cG $ and fix any $t\in\bbR$. 
Then the random variable $ Y $ defined in \Cref{eqn:Y} satisfies $ Y\in\cG $ and $ S_X(t) = S_{Y}(t) $. 
\end{lemma}

\begin{proof}
By the law of total expectation, $ \expt{Y} = \expt{\expt{X \mid \one_{\cE}}} = \expt{X} = 0 $. 
Moreover, for any $\lambda\in\bbR$, 
\begin{align}
    M_Y(\lambda) &= p e^{\lambda \mu_+} + (1-p) e^{\lambda \mu_-}
    = \expt{ \exp\paren{\lambda \expt{X \mid \one_{\cE}}} }
    \le \expt{ \expt{ \exp\paren{\lambda X} \mid \one_{\cE} } }
    = M_X(\lambda) \le e^{\lambda^2/2} , \notag 
\end{align}
where the first inequality is by applying the conditional Jensen's inequality $ f(\expt{X\mid Z}) \le \expt{f(X) \mid Z} $ to the convex function $ f(x) = e^{\lambda x} $ and the random variable $ Z = \one_{\cE} $. 
Therefore $ Y\in\cG $. 

Finally, 
\begin{align}
    S_X(t) &= \expt{(X-t)_+} = p \expt{X - t \mid \cE} 
    % = p (\expt{X \mid \cE} - t)
    = p(\mu_+ - t)
    = \expt{(Y - t)_+}
    = S_Y(t) , \notag 
\end{align}
as desired. 
\end{proof}

By \Cref{lem:Y}, for the purpose of understanding $ S_X(t) $, we only need to consider $ Y $ constructed in \Cref{eqn:Y}. 
We make the following change of variable: $ \nu (1 - p) = \mu_+ $. 
Since $ \expt{Y} = 0 $, we know that $ p \mu_+ + (1-p) \mu_- = 0 $. 
Then $ p \nu (1-p) + (1-p) \mu_- = 0 $, i.e., $ \mu_- = - p\nu $. 
Also, $ \mu_+>0, \nu>0 $ since $ \mu_+ \ge \mu_- $. 
To summarize, we have 
\begin{align}
    & -p\nu = \mu_- \le \min\{t,0\} \le \max\{t,0\} \le \mu_+ = (1-p)\nu . \notag 
\end{align}
We can write the law of $Y$ in a more symmetrical form $ p \delta_{(1-p)\nu} + (1-p) \delta_{- p\nu} $. 
The MGF of $Y$ is given by 
\begin{align}
    M_Y(\lambda) &= p e^{\lambda (1-p) \nu} + (1-p) e^{-\lambda p \nu} . \notag 
\end{align}

First consider the case $ p\ne1/2 $ and set $ \beta \coloneqq \log( (1-p)/p ) $. 
Then 
\begin{align}
    M_Y(2 \beta / \nu) &= p e^{2 \beta (1-p)} + (1-p) e^{- 2 \beta p}
    = p \paren{\frac{1-p}{p}}^{2(1-p)} + (1-p) \paren{\frac{1-p}{p}}^{-2p} \notag \\
    &= \brack{p \paren{\frac{1-p}{p}}^2 + (1-p)} \paren{\frac{1-p}{p}}^{-2p}
    = \brack{\frac{1-p}{p} + 1} (1-p) \paren{\frac{1-p}{p}}^{-2p} \notag \\
    &= \paren{\frac{1-p}{p}}^{1-2p}
    = e^{(1-2p) \beta} . \notag 
\end{align}
Since $ Y\in\cG $, this is upper bounded by 
\begin{align}
    \exp\paren{ (2\beta/\nu)^2/2 } &= e^{2\beta^2/\nu^2} . \notag 
\end{align}
Combining the preceding two displays, we have
\begin{align}
    \nu^2 &\le \frac{2\beta}{1 - 2p} , \label{eqn:nu} 
\end{align}
whose RHS is nonnegative since $ 1 - 2p $ and $ \beta $ have the same sign. 

Now consider $ p=1/2 $. 
Recall that $ M_Y(\lambda) \le e^{\lambda^2/2} $ for all $ \lambda\in\bbR $. 
Taking Taylor series at $0$ on both sides and using the fact $ \expt{Y} = 0 $, we have
\begin{align}
    1 + \frac{\lambda^2}{2} \expt{Y^2} + O(\lambda^3)
    &\le 1 + \frac{\lambda^2}{2} + O(\lambda^4) . \notag 
\end{align}
% Differentiating both sides of the inequality 
% \begin{align}
%     M_Y(\lambda) &= \sum_{k=0}^\infty \frac{\lambda^k \expt{Y^k}}{k!} \le e^{\lambda^2/2} \notag 
% \end{align}
% twice, we have 
% \begin{align}
%     \expt{Y^2} &\le (\lambda^2 + 1) e^{\lambda^2/2} . \notag 
% \end{align}
Taking $ \lambda \downarrow 0 $ yields $ \expt{Y^2} \le 1 $. 
An explicit calculation shows that 
\begin{align}
    \expt{Y^2} = p ( (1-p)\nu )^2 + (1-p) (p \nu)^2 
    = p(1-p) \nu^2 . \notag 
\end{align}
At $ p=1/2 $, we conclude $ \nu^2\le4 $. 
Note that $4$ is precisely the limit of the RHS of \Cref{eqn:nu} as $p\to1/2$. 
Therefore, \Cref{eqn:nu} holds for all $ p\in(0,1) $, where the upper bound is interpreted as its limit at the removable singularity $ p=1/2 $. 

Using \Cref{eqn:nu} and $ \expt{Y} = 0 $, recalling the definition of $\beta$, we have 
\begin{align}
    p \cdot (1-p)\nu &= (1-p) \cdot p\nu
    \le p(1-p) \sqrt{\frac{2\log\frac{1-p}{p}}{1-2p}} . \label{eqn:ineq0} 
\end{align}
By \Cref{lem:I}, the RHS above can be further upper bounded as follows
\begin{align}
    p(1-p) \sqrt{\frac{2\log\frac{1-p}{p}}{1-2p}}
    &\le \sqrt{\frac{\pi}{2}} I(p) . \label{eqn:ineq} 
\end{align}

Back to the proof of \Cref{eqn:TODO}, let us study its RHS. 
An explicit calculation shows that for any $ b>0 $, 
\begin{align}
&&
    S_{b\,G}(t) &= \frac{b}{\sqrt{2\pi}} \exp\paren{-\frac{t^2}{2b^2}} - \frac{t}{2} \erfc\paren{\frac{t}{\sqrt{2} \,b }} , & 
    S_{b\,G}'(t) &= -\frac{1}{2} \erfc\paren{\frac{t}{\sqrt{2} \, b}} . & 
& \label{eqn:explicit} 
\end{align}
Since $ t \mapsto (x-t)_+ $ is convex for any $x\in\bbR$, $ S_{b\,G}(t) $ is convex. 
Also, it trivially holds that 
\begin{align}
&&
    S_{b\,G}(t) &\ge 0 , & 
    S_{b\,G}(t) &\ge \expt{b\,G-t} = -t . & 
& \label{eqn:trivial}
\end{align}

Now take $ x = \Phi^{-1}(1-p) $. 
Using \Cref{eqn:explicit}, we have 
\begin{align}
    S_{b\,G}(b \, x) &= b \, \varphi(x) - b \, x (1 - \Phi(x)) 
    = b \, \varphi(x) - b \, x \, p , \notag \\
    S_{b\,G}'(b \, x) &= - (1 - \Phi(x)) = -p . \notag 
\end{align}
So the tangent line (parametrized by $t$) of $ S_{b\,G} $ at $ t = b\,x $ is 
\begin{align}
    S_{b\,G}'(b\,x) \cdot (t - b\,x) + S_{b\,G}(b\,x)
    &= b\,\varphi(x) - pt = b I(p) - pt , \notag 
\end{align}
where the last step follows since 
\begin{align}
    I(1-p) &= \varphi(\Phi^{-1}(1-p)) = \varphi(-\Phi^{-1}(p)) = \varphi(\Phi^{-1}(p)) = I(p) . \notag 
\end{align}
By convexity of $ S_{b\,G} $, the graph of $ S_{b\,G} $ lies above the tangent line passing through $ (b\,x, S_{b\,G}(b\,x)) $, and hence
\begin{align}
    S_{b\,G}(t) &\ge b I(p) - pt . \label{eqn:lb} 
\end{align}
From \Cref{eqn:ineq0,eqn:ineq}, 
\begin{align}
    p (1-p) \nu \le \sqrt{\frac{\pi}{2}} I(p) . \notag 
\end{align}
Using this back in \Cref{eqn:lb} and taking $ b = \sqrt{\pi/2} $, we get 
\begin{align}
    S_{\sqrt{\pi/2}\,G}(t) &\ge p ((1-p) \nu - t) . \notag 
\end{align}
Combining this with \Cref{eqn:trivial}, we have 
\begin{align}
    S_{\sqrt{\pi/2}\,G}(t) &\ge \max\brace{ 0 , -t , p( (1-p)\nu - t ) } . \label{eqn:S}
\end{align}

Next, we study $ S_Y(t) $. 
Noting that $ -p\nu \le 0\le (1-p)\nu $, we have the following explicit expressions of $ S_Y(t) $: 
\begin{itemize}
    \item for $ t\le-p\nu $, we have
    \begin{align}
        S_{Y}(t) &= (1-p)(-p\nu - t) + p( (1-p)\nu - t ) = -t ; \notag 
    \end{align}

    \item for $ -p\nu\le t\le (1-p)\nu $, we have $ S_{Y}(t) = p( (1-p)\nu - t ) $; 

    \item for $ t\ge(1-p)\nu $, we have $ S_{Y}(t) = 0 $. 
\end{itemize}
It is not hard to verify that these cases can be collectively written as: 
\begin{align}
    S_Y(t) &= \max\brace{ -t, p( (1-p)\nu - t ) , 0 } . \notag 
\end{align}
Contrasting this with \Cref{eqn:S}, we conclude \Cref{eqn:TODO}. 
This therefore proves that $ \cmgf \le \sqrt{\pi/2} $. 

Finally, we prove a matching lower bound on $ \cmgf $ by exhibiting a random variable and a convex function. 
Let $ X \sim \unif(\{-1,1\}) $. 
It can be easily verified that $ M_X(\lambda) = \cosh(\lambda) \le e^{\lambda^2/2} $ for all $\lambda\in\bbR$, so $ X\in\cG $. 
Take the convex function $ f(x) = \abs{x} $. 
Then $ \cmgf $ must satisfy 
\begin{align}
    1 &= \expt{\abs{X}} \le \expt{\abs{\cmgf G}} = \cmgf \sqrt{2/\pi} . \label{eqn:emerge} 
\end{align}
Thus $ \cmgf \ge \sqrt{\pi/2} $. 
This completes the proof of \Cref{thm:cmgf}. 

\section{Concluding remarks}

This paper identifies the smallest constant $ \cmgf $ to be $1/\expt{\abs{G}}$ such that any $1$-sub-Gaussian random variable in the sense of \Cref{eqn:subG} is dominated by $ \cmgf G $ in convex order.
The proof crucially relies on an equivalent characterization of convex domination (\Cref{prop:hinge}) which is in turn due to the hinge representation \Cref{eqn:hinge} of univariate convex functions. 
% To the best of our knowledge, a natural analogue of such representations in higher dimensions is not available. 
In higher dimensions, a natural analogue of \Cref{eqn:hinge} for $ f\colon\bbR^d\to\bbR $ convex reads
\begin{align}
    f(\tau \xi) &= a_\xi \tau + b_\xi + \int_0^\infty \max\brace{\tau - t, 0} \, \mu_\xi(\dd t) , \label{eqn:hinged}
\end{align}
where $ \tau\ge0 $ and $ \xi\in\bbS^{d-1} $. 
The triple $ (a_\xi,b_\xi,\mu_\xi) $ can be identified by Taylor's theorem as $ (f_\xi'(0), f_\xi(0), f_\xi'') $ where $ f_\xi \colon \tau \mapsto f(\tau \xi) $ and $ f_\xi'' $ is interpreted as a Schwartz distribution. 
Applying \Cref{eqn:hinged} to an arbitrary $ x\in\bbR^d \setminus \{0_d\} $ with $ \tau = \normtwo{x} $ and $ \xi = x/\normtwo{x} $, we get a potentially nonconvex function $ x \mapsto \max\brace{ \normtwo{x} - t, 0 } f''_{x/\normtwo{x}}(t) $. 
This prevents us from reducing the verification of convex domination to simple convex hinge functions. 
Our approach (notably the upper bound) therefore does not generalize to $d\ge2$ and leaves open the characterization of the sharp sub-Gaussian comparison constant under either \Cref{eqn:notion1} or \Cref{eqn:notion2}, the latter of which hopefully admits a more interpretable answer. 
The sharpness of the same constants for $d=1$ identified in \Cref{thm:cmgf} under \Cref{eqn:notion2} and in \cite[Theorem 1]{Davis_Power} under \Cref{eqn:notion1} is not expected to hold for $d\ge2$ (see \Cref{thm:highd}). 

% Inspecting the constructions of random variable / vector and convex functions for the lower bounds in \Cref{thm:cmgf,thm:2d} (see \Cref{eqn:emerge,eqn:X,eqn:f}), we expect it to be instructive to investigate their natural high-dimensional analogues. 
% Specifically, let $ u_0, \cdots, u_d \in \bbS^{d-1} $ be such that $ \inprod{u_i}{u_j} = -1/d $ for all $ i\ne j $. 
% Such vectors can be obtained by rigidly embedding the vertices of a regular simplex in $ \bbR^{d+1} $ into $ \bbR^d $ and shifting their centroid to the origin. 
% Let $ X $ be uniformly distributed on $ \brace{ ru_j }_{j=0}^d $ where $r$ is the largest constant such that $ X $ remains $ 1 $-sub-Gaussian. 
% Define the convex function $ f \colon \bbR^d\to\bbR $ as $ f(x) = \max_{0\le j\le d} \inprod{u_j}{x} $. 
% Analyzing $ X $ and $ f $ will lead to a lower bound on $ \cmgf(d) $. 
% However, we are not sufficiently confident to make conjectures regarding the extremal properties of such examples. 

Finally, regarding the key technical input \Cref{eqn:I}, we note that though it lower bounds the Gaussian isoperimetric function by a rather simple and convenient elementary function (see \Cref{rk:I}), our proof is purely analytical and offers no geometric insights. 
Understanding the geometric meaning and consequences of \Cref{eqn:I} may shed light on the high-dimensional case. 

%%%%%%%%%%%%%%%%%%%%%%%%%%%%%%
%%%%%%%%%%%%%%%%%%%%%%%%%%%%%%
%%%%%%%%%%%%%%%%%%%%%%%%%%%%%%

% \bibliographystyle{alpha}
% \bibliography{ref} 

\renewcommand*{\bibfont}{\normalfont\small}
\printbibliography

\appendix 

\section{Proof of \texorpdfstring{\Cref{lem:I}}{}}
\label{app:ineq}

It is easy to check that both sides of \Cref{eqn:I} are symmetric around $1/2$ in the sense that they are invariant under the change of variable $ p\mapsto1-p $. 
As $ p\downarrow0 $, both sides of \Cref{eqn:I} have limits $0$. 
So we can without loss of generality assume $ 0<p\le1/2 $. 
Let 
\begin{align}
&&
    x &\coloneqq \Phi^{-1}(1-p) \ge 0 , & 
    u &\coloneqq 1-2p = 2\Phi(x) - 1 = \erf(x/\sqrt{2}) \in [0,1) . & 
& \label{eqn:xu} 
\end{align}
Then 
\begin{align}
&&
    p(1-p) &= \frac{1-u^2}{4} , & 
    \log\frac{1-p}{p} &= 2 \arctanh(u) , &
    \sqrt{\pi/2}\, I(p) &= \sqrt{\pi/2}\, \varphi(x) = e^{-x^2/2}/2 . & 
& \notag 
\end{align}
% where we recall $ \arctanh(x) = \frac{1}{2} \log\paren{\frac{1+x}{1-x}} $. 
Now \Cref{eqn:I} becomes
\begin{align}
    (1-u^2) \sqrt{\frac{\arctanh(u)}{u}} &\le e^{-x^2/2} . \label{eqn:ineq_TODO} 
\end{align}

At $ x = 0 $, we have $ u = 0 $ and $ \arctanh(u) / u \to 1 $ as $ u\to0 $. 
Therefore the LHS of \Cref{eqn:ineq_TODO} equals $1$ by continuity, and the RHS is also equal to $1$. 
We assume henceforth $ x>0 $, or equivalently $ u>0 $. 

Let 
\begin{align}
    t &\coloneqq \arctanh(u) > 0 , \notag
\end{align}
so that $ u = \tanh(t) $. 
In the rest of the proof, we treat all variables as functions of $t$.
For instance, we will write $u(t)$ for $u$. 
Define 
\begin{align}
    R(t) &\coloneqq 4 \log(\cosh(t)) - \log\paren{\frac{t}{\tanh(t)}}
    = 4 \log(\cosh(t)) - \log\paren{\frac{t}{u(t)}} . \label{eqn:R} 
\end{align}
By $ \sech(t)^2 + \tanh(t)^2 = 1 = \sech(t) \cosh(t) $, we have 
\begin{align}
    1 - u(t)^2 &= \frac{1}{\cosh(t)^2} , \notag 
\end{align}
and hence
\begin{align}
    \exp\paren{-\frac{R(t)}{2}}
    &= \exp\paren{ -2 \log(\cosh(t)) + \frac{1}{2} \log\paren{\frac{t}{u(t)}} }
    = \frac{1}{\cosh(t)^2} \sqrt{\frac{t}{u(t)}} \notag \\
    &= (1-u(t)^2) \sqrt{\arctanh(u(t))/u(t)} , \label{eqn:expR}
\end{align}
which is precisely the LHS of \Cref{eqn:ineq_TODO}. 
Therefore, the desired inequality \Cref{eqn:ineq_TODO} is equivalent to $ \exp\paren{-R(t)/2} \le \exp\paren{-x^2/2} $, or $ R(t) \ge x^2 $. 
Since $ \erf $ is strictly increasing on $ (0,\infty) $, this is further equivalent to 
\begin{align}
    \erf\paren{\sqrt{\frac{R(t)}{2}}} &\ge \erf\paren{\frac{x}{\sqrt{2}}} = u(t) = \tanh(t) . \notag 
\end{align}
Let 
\begin{align}
    H(t) &\coloneqq \erf\paren{\sqrt{\frac{R(t)}{2}}} - \tanh(t) . \label{eqn:H}
\end{align}
By the above reductions, the sought result \Cref{eqn:ineq_TODO} follows from the lemma below. 

\begin{lemma}
\label{lem:H}
Consider the function $H$ defined in \Cref{eqn:H}. 
Then $H$ is nonnegative on $ (0,\infty) $. 
\end{lemma}

To prove \Cref{lem:H}, we need a sequence of preparatory lemmas. 

\begin{lemma}
\label{lem:R}
Consider the function $R$ defined in \Cref{eqn:R}. 
Then for all $ t>0 $, 
\begin{align}
&&
    R'(t) &= 3 u(t) + 1/u(t) - 1/t > 0 , & 
    \lim_{t\downarrow0} R(t) &= 0 , & 
    R(t) &> 0 . & 
& \label{eqn:R'}
\end{align}
\end{lemma}

\begin{proof}
By basic hyperbolic trigonometry,  
\begin{align}
    \frac{\dd}{\dd t} \tanh(t) &= \sech(t)^2 = 1 - \tanh(t)^2 , \notag 
\end{align}
which, in our notation, becomes
\begin{align}
    u'(t) &= 1 - u(t)^2 . \label{eqn:u'}
\end{align}
Also, 
\begin{align}
    \frac{\dd}{\dd t} \log\paren{\frac{t}{u(t)}}
    &= \frac{1}{t} - \frac{u'(t)}{u(t)} . \notag 
\end{align}
Using the preceding two displays and the definition of $R$ in \Cref{eqn:R}, we can explicitly compute $ R' $ as
\begin{align}
    R'(t) &= 4\frac{\sinh(t)}{\cosh(t)} - \paren{ \frac{1}{t} - \frac{u'(t)}{u(t)} }
    = 4 u(t) - 1/t + (1-u(t)^2)/u(t)
    = 3 u(t) - 1/t + 1/u(t) , \notag 
\end{align}
as claimed. 
Since $ 0 < \tanh(t) < t $ for every $t>0$, we immediately have $ R'(t) > 0 $. 
Also, it is easy to verify that $ R(t) \to 0 $ as $ t\downarrow0 $. 
Thus, $ R $ is positive on $ (0,\infty) $. 
\end{proof}

Now, define 
\begin{align}
    M(t) &\coloneqq R'(t) \sqrt{\frac{t}{u(t) R(t)}} . \label{eqn:M} 
\end{align}

The key step is to show monotonicity of $M$ which is achieved in the next lemma. 

\begin{lemma}
\label{lem:M}
Consider the function $M$ defined in \Cref{eqn:M}. 
Then $ M $ is strictly decreasing on $ (0,\infty) $. 
\end{lemma}

\begin{proof}
It is equivalent to proving that 
\begin{align}
    V(t) \coloneqq M(t)^2 &= \frac{t R'(t)^2}{u(t) R(t)} \label{eqn:V} 
\end{align}
is strictly decreasing. 
Our strategy is to show that $ V' $ is strictly negative on $ (0,\infty) $. 
The proof is divided into several steps. 

\paragraph{Explicit expression of $ V' $.}
Taking the derivative of the logarithm of both sides of the defining equation \Cref{eqn:V} for $V$, we obtain 
\begin{align}
    \frac{V'(t)}{V(t)} &= \frac{1}{t} + 2\frac{R''(t)}{R'(t)} - \frac{u'(t)}{u(t)} - \frac{R'(t)}{R(t)} . \label{eqn:V'} 
\end{align}
Defining 
\begin{align}
    B(t) &\coloneqq 2 R''(t) + R'(t) \paren{\frac{1}{t} - \frac{u'(t)}{u(t)}} , \label{eqn:B} 
\end{align}
we can write the first three terms as 
\begin{align}
    \frac{1}{t} + 2\frac{R''(t)}{R'(t)} - \frac{u'(t)}{u(t)}
    &= \frac{1}{R'(t)} \paren{ 2R''(t) + R'(t) \paren{ \frac{1}{t} - \frac{u'(t)}{u(t)} } }
    = \frac{B(t)}{R'(t)} . \notag 
\end{align}
Plugging this back to \Cref{eqn:V'} and using \Cref{eqn:V}, we have 
\begin{align}
    V'(t) &= \frac{t R'(t)^2}{u(t) R(t)} \paren{ \frac{B(t)}{R'(t)} - \frac{R'(t)}{R(t)} }
    = \frac{t R'(t)}{u(t) R(t)^2} \paren{ B(t) R(t) - R'(t)^2 } . \notag 
\end{align}
Since $ t,R'(t),u(t) $ are all positive, the sign of $ V'(t) $ is the sign of 
\begin{align}
    B(t)R(t) - R'(t)^2 \label{eqn:BR-R'2}
\end{align}
which we will prove to be negative. 
To this end, we first provide an explicit expression of $ B $. 

\paragraph{Explicit expressions of $B$ and $ B' $.}
Let us first compute $ R''(t) $. 
Taking the derivative of \Cref{eqn:R'} and using \Cref{eqn:u'}, we have 
\begin{align}
    R''(t) &= 3u'(t) - \frac{u'(t)}{u(t)^2} + \frac{1}{t^2}
    = 3(1 - u(t)^2) - \frac{1 - u(t)^2}{u(t)^2} + \frac{1}{t^2}
    = 4 - 3 u(t)^2 - \frac{1}{u(t)^2} + \frac{1}{t^2} . \label{eqn:R''} 
\end{align}
Plugging this and \Cref{eqn:R'} into $B$ in \Cref{eqn:B}, we get 
\begin{align}
    B(t) &= 2 \paren{ 4 - 3 u(t)^2 - \frac{1}{u(t)^2} + \frac{1}{t^2} }
    + \paren{ 3u(t) + \frac{1}{u(t)} - \frac{1}{t} } \paren{ \frac{1}{t} - \frac{1 - u(t)^2}{u(t)} } \notag \\
    &= 6 - 3u(t)^2 - \frac{3}{u(t)^2} + \frac{2u(t)}{t} + \frac{2}{t u(t)} + \frac{1}{t^2} , \label{eqn:expr_B} 
\end{align}
where the last equality is by expanding all products and performing elementary simplifications. 
Differentiating $B$ gives
\begin{align}
    B'(t) &= -6u(t)u'(t) + \frac{6u'(t)}{u(t)^3} + \frac{2u'(t)}{t} - \frac{2u(t)}{t^2} - \frac{2u'(t)}{tu(t)^2} - \frac{2}{t^2u(t)} - \frac{2}{t^3} \notag \\
    &= -6u(t)(1 - u(t)^2) + \frac{6(1 - u(t)^2)}{u(t)^3} + \frac{2(1 - u(t)^2)}{t} - \frac{2u(t)}{t^2} - \frac{2(1 - u(t)^2)}{tu(t)^2} - \frac{2}{t^2u(t)} - \frac{2}{t^3} , \label{eqn:B'}
\end{align}
where the second line uses \Cref{eqn:u'}. 

\paragraph{Nonnegativity of $B$.}
Next, we show $ B(t)>0 $ for all $t>0$. 
Introduce additional variables
\begin{align}
&&
    r(t) &\coloneqq \frac{t}{u(t)} , & 
    z(t) &\coloneqq r(t) (1 - u(t)^2) = \frac{t(1 - u(t)^2)}{u(t)} . &
& \label{eqn:rz} 
\end{align}
Since $ u(t) < t $, we have 
\begin{align}
    r(t) &> 1 . \label{eqn:r>1}
\end{align}
Also, by definition, 
\begin{align}
    z(t) &= \frac{t (1 - \tanh(t)^2)}{\tanh(t)}
    = \frac{t \sech(t)^2}{\tanh(t)}
    = \frac{t / \cosh(t)^2}{\sinh(t) / \cosh(t)}
    = \frac{t}{\sinh(t) \cosh(t)}
    = \frac{2t}{\sinh(2t)} . \label{eqn:expr_z} 
\end{align}
Since $ \sinh(x) > x $ for all $ x>0 $, we have 
\begin{align}
    z(t) \in (0,1) . \label{eqn:z<1}
\end{align}

Using \Cref{eqn:expr_B}, let us compute $ t^2 B(t) $: 
\begin{align}
    t^2 B(t) &= 6t^2 - 3u(t)^2t^2 - \frac{3t^2}{u(t)^2} + 2u(t)t + \frac{2t}{u(t)} + 1 \notag \\
    &= 6r(t)^2u(t)^2 - 3u(t)^4r(t)^2 - 3r(t)^2 + 2u(t)^2r(t) + 2r(t) + 1 , \notag 
\end{align}
where in obtaining the second line, we use $ r(t) $ to eliminate $ t $ factors. 
On the other hand, we have 
\begin{align}
    1 + 4 r(t) - 2z(t) - 3 z(t)^2
    &= 1 + 4r(t) - 2r(t) (1 - u(t)^2)
    -3r(t)^2 (1 - u(t)^2)^2 \notag \\
    &= 1 + 2r(t) + 2r(t) u(t)^2 - 3r(t)^2 + 6r(t)^2u(t)^2 - 3r(t)^2u(t)^4 , \notag 
\end{align}
which coincides with the expression of $ t^2 B(t) $ above. 
So 
\begin{align}
    t^2B(t) &= 1 + 4 r(t) - 2z(t) - 3 z(t)^2
    > 5 - 2z(t) - 3 z(t)^2
    % = (1 - z(t)) (5 + 3z(t)) 
    , \notag 
\end{align}
where the inequality is by \Cref{eqn:r>1}. 
By \Cref{eqn:z<1}, the quadratic in $z(t)$ on the RHS is positive. 
This proves 
\begin{align}
    B(t) &> 0 \label{eqn:B>0}
\end{align}
for all $t>0$. 

At this point, recalling the goal of proving negativity of \Cref{eqn:BR-R'2}, we define
\begin{align}
    D(t) &\coloneqq \frac{R'(t)^2}{B(t)} - R(t) . \label{eqn:D}
\end{align}
Since $ B(t)>0 $ just shown, it suffices to show that $ D(t) > 0 $ for all $t>0$. 

\paragraph{Limit of $D$ at $0$.}
We claim that
\begin{align}
    \lim_{t\downarrow0} D(t) &= 0 . \label{eqn:limD}
\end{align}
Indeed, the standard Taylor expansion around $t=0$ yields $ u(t) = \tanh(t) = t - t^3/3 + O(t^5) $. 
Using this with the explicit expressions of $ R(t),R'(t),B(t) $ in \Cref{eqn:R,eqn:R',eqn:expr_B}, we conclude \Cref{eqn:limD}. 

\paragraph{Monotonicity of $D$.}
We now prove $ D'(t) > 0 $ for all $ t>0 $. 
Differentiating $D$ gives
\begin{align}
    D'(t) &= \frac{R'(t)}{B(t)^2} ( 2R''(t) B(t) - R'(t) B'(t) - B(t)^2 ) . \notag 
\end{align}
Since both $ B $ and $ R' $ are positive on $ (0,\infty) $ (see \Cref{eqn:B>0} and \Cref{lem:R}), it remains to show 
\begin{align}
    2R''(t) B(t) - R'(t) B'(t) - B(t)^2 &> 0 . \label{eqn:D'_TODO}
\end{align}
Substituting \Cref{eqn:R'',eqn:R',eqn:B,eqn:B'} in the LHS above for $ R'(t),B(t),R''(t),B'(t) $, we obtain 
\begin{align}
    2R''(t) B(t) - R'(t) B'(t) - B(t)^2
    &= \frac{1}{t^4 u(t)^4} \big[ 
        - 9t^4u(t)^8 
        - 12t^4u(t)^6 
        + 42t^4u(t)^4 
        - 12t^4u(t)^2
        \notag \\ &\qquad\qquad\quad 
        - 9t^4 
        + 12t^3u(t)^7 
        - 24t^3u(t)^5
        - 4t^3u(t)^3
        + 16t^3u(t)
        \notag \\ &\qquad\qquad\quad 
        - 6t^2u(t)^6
        + 12t^2u(t)^4
        - 6t^2u(t)^2
        + 4tu(t)^5
        - u(t)^4
    \big] . \notag 
\end{align}
The terms in the brackets can be factorized as $ -[t (3u(t)^2 + 1) - u(t)] \cdot P(t, u(t)) $ where the bivariate polynomial $P$ is given by 
\begin{align}
    P(t,u) &\coloneqq 
    3t^3u^6
    + 3t^3u^4
    - 15t^3u^2
    + 9t^3
    - 3t^2u^5
    + 10t^2u^3
    - 7t^2u 
    + tu^4
    - tu^2
    - u^3 . 
    \notag 
\end{align}
Therefore, 
\begin{align}
    2R''(t) B(t) - R'(t) B'(t) - B(t)^2
    &= - \frac{t (3u(t)^2 + 1) - u(t)}{t^4 u(t)^4} P(t,u(t)) . \label{eqn:P_TODO} 
\end{align}
Since $ t (3u(t)^2 + 1) - u(t) = (t - u(t)) + 3tu(t)^2 > 0 $ by $ t>u(t) $, it is enough to prove that $ P(t,u(t)) < 0 $ for all $t>0$. 

\paragraph{Negativity of $P$.}
We again work with the auxiliary variables $ r(t), z(t) $ in \Cref{eqn:rz}. 
Let us consider $ P/u^3 $
 % $ P(t,u(t))/u(t)^3 $ 
and use $ r(t) $ to eliminate all $ t $ factors: 
\begin{align}
    \frac{P(t,u(t))}{u(t)^3}
    &= 3t^3u(t)^3
    + 3t^3u(t)
    - \frac{15t^3}{u(t)}
    + \frac{9t^3}{u(t)^3}
    - 3t^2u(t)^2
    + 10t^2
    - \frac{7t^2}{u(t)^2} 
    + tu(t)
    - \frac{t}{u(t)}
    - 1 \notag \\
    &= 3r(t)^3u(t)^6
    + 3r(t)^3u(t)^4
    - 15r(t)^3u(t)^2
    + 9r(t)^3
    \label{eqn:l1} \\ &\quad
    - 3r(t)^2u(t)^4
    + 10r(t)^2u(t)^2
    - 7r(t)^2
    \label{eqn:l2} \\ &\quad
    + r(t)u(t)^2
    - r(t)
    - 1 . \label{eqn:l3}
\end{align}
Further introduce $ q(t) \coloneqq 1 - u(t)^2 $. 
Then trivially, 
\begin{align}
&&
    z(t) &= r(t) q(t) , &
    q(t) &= z(t) / r(t) , & 
    u(t)^2 &= 1 - q(t) . & 
& \notag
\end{align}
We use these auxiliary variables to simplify the RHS of $ P/u^3 $. 
The first line \Cref{eqn:l1} can be written as 
\begin{align}
    3r(t)^3u(t)^6
    + 3r(t)^3u(t)^4
    - 15r(t)^3u(t)^2
    + 9r(t)^3
    &= 3r(t)^3 \brack{
        (1 - q(t))^3
        + (1 - q(t))^2
        - 5(1 - q(t))
        + 3
    } \notag \\
    &= 3r(t)^3 \brack{4q(t)^2 - q(t)^3}
    = 12 r(t) z(t)^2 - 3 z(t)^3 . \notag 
\end{align}
Similarly, the second line \Cref{eqn:l2} becomes
\begin{align}
    - 3r(t)^2u(t)^4
    + 10r(t)^2u(t)^2
    - 7r(t)^2
    &= r(t)^2 \brack{
        - 3(1 - q(t))^2
        + 10(1 - q(t))
        - 7
    } \notag \\
    &= r(t)^2 \brack{ -4q(t) - 3q(t)^2 }
    = -4r(t)z(t) - 3z(t)^2 . \notag 
\end{align}
Finally, for the third line \Cref{eqn:l3}, we have 
\begin{align}
    r(t) u(t)^2 - r(t) - 1
    &= r(t) \brack{(1 - q(t)) - 1} - 1
    = -r(t)q(t) - 1 = -z(t) - 1 . \notag 
\end{align}
Summing up the preceding three displays, 
\begin{align}
    \frac{P(t,u(t))}{u(t)^3}
    &= r(t)z(t) (12z(t)-4)-3z(t)^3-3z(t)^2-z(t)-1 . \label{eqn:P_u3}
\end{align}

\paragraph{An elementary inequality.}
Before proceeding, let us establish an elementary inequality 
\begin{align}
    2r(t)z(t) &< 1 + z(t) . \label{eqn:rz_TODO}
\end{align}
By \Cref{eqn:rz}, $ r(t)z(t) $ is explicitly given by 
\begin{align}
    r(t)z(t)
    &= \frac{t}{u(t)} \cdot \frac{t(1 - u(t)^2)}{u(t)}
    = t^2 \cdot \frac{1-\tanh(t)^2}{\tanh(t)^2}
    = t^2 \cdot \frac{\sech(t)^2}{\tanh(t)^2}
    = \frac{t^2}{\sinh(t)^2} . \notag 
\end{align}
Also, from \Cref{eqn:expr_z}, 
\begin{align}
    z(t) &= \frac{t}{\sinh(t) \cosh(t)} . \notag 
\end{align}
So \Cref{eqn:rz_TODO} is equivalent to 
\begin{align}
    \frac{2t^2}{\sinh(t)^2} &< 1 + \frac{t}{\sinh(t) \cosh(t)} , \notag 
\end{align}
or 
\begin{align}
    2t^2 \cosh(t) &< \sinh(t)^2 \cosh(t) + t \sinh(t) . \notag 
\end{align}
Defining 
\begin{align}
    E(t) &\coloneqq \sinh(t)^2 \cosh(t) + t \sinh(t) - 2t^2 \cosh(t) , \notag 
\end{align}
we will show $ E(t) > 0 $ for all $ t>0 $. 
Using the identity
\begin{align}
    \sinh(t)^2 \cosh(t) &= \frac{\cosh(3t) - \cosh(t)}{4} , \notag 
\end{align}
we write $ E(t) $ as 
\begin{align}
    E(t) &= \frac{\cosh(3t) - \cosh(t)}{4} + t \sinh(t) - 2t^2 \cosh(t) . \notag 
\end{align}
Expanding all three terms into power series at $0$ (which converge on the entire real line), we have 
\begin{align}
    E(t) &= \sum_{n=0}^\infty \frac{(3^{2n} - 1) t^{2n}}{4 \cdot (2n)!} 
    + \sum_{n=0}^\infty \frac{t^{2n+2}}{(2n+1)!}
    - \sum_{n=0}^\infty\frac{2 \cdot t^{2n+2}}{(2n)!} \notag \\
    &= \sum_{n=1}^\infty \frac{(9^{n} - 1) \cdot t^{2n}}{4 \cdot (2n)!} 
    + \sum_{n=1}^\infty \frac{t^{2n}}{(2n-1)!}
    - \sum_{n=1}^\infty\frac{2 \cdot t^{2n}}{(2n-2)!} \notag \\
    &= \sum_{n = 3}^\infty \frac{9^n - 32n^2 + 24n - 1}{4 \cdot(2n)!} t^{2n} , \notag 
\end{align}
where the last equality follows since the coefficient in front of $ t^{2n} $ vanishes for $ n\in\{1,2\} $. 
Moreover, it is easy to verify that this coefficient is positive for all $n\ge3$. 
This allows us to conclude $ E(t) > 0 $ and therefore the inequality \Cref{eqn:rz_TODO}. 

\paragraph{Concluding negativity of $P$.}
Recall that our mission is to show negativity of $ P(t,u(t)) $. 
Since $ u(t)>0 $, we only need to show negativity of the RHS of \Cref{eqn:P_u3}. 

Recalling \Cref{eqn:z<1}, we consider two cases. 
\begin{enumerate}
    \item 
    If $ 0<z(t)\le1/3 $, then by \Cref{eqn:r>1}, $ r(t) z(t) (12z(t)-4) \le 0 $. 
    Thus 
    \begin{align}
        \frac{P(t,u(t))}{u(t)^3} &\le -3z(t)^3 - 3z(t)^2 - z(t) - 1 , \notag 
    \end{align}
    the RHS of which is negative. 

    \item 
    If $ 1/3<z(t)<1 $, using \Cref{eqn:rz_TODO}, we have 
    \begin{align}
        r(t) z(t) (12z(t)-4)
        &\le \frac{1}{2} (1 + z(t)) (12z(t)-4) . \notag 
    \end{align}
    Thus, 
    \begin{align}
        \frac{P(t,u(t))}{u(t)^3} &\le \frac{1}{2} (1 + z(t)) (12z(t)-4) -3z(t)^3 - 3z(t)^2 - z(t) - 1
        = -3z(t)^3 + 3z(t)^2 + 3z(t) - 3 , \notag 
    \end{align}
    the RHS of which is again negative. 
\end{enumerate}
This completes the proof of $ P(t,u(t))<0 $. 
Backtracking to \Cref{eqn:P_TODO}, we therefore establish \Cref{eqn:D'_TODO}, which in turn implies $ D'(t)>0 $. 
Combining this with \Cref{eqn:limD}, we obtain positivity of $D$ which is equivalent to negativity of \Cref{eqn:BR-R'2} and $ V' $. 
By \Cref{eqn:V}, we finally conclude that $M$ is decreasing on $(0,\infty)$. 
\end{proof}

We are in a position to prove \Cref{lem:H}. 

\begin{proof}[Proof of \Cref{lem:H}.]
Recall from \Cref{eqn:H} that
\begin{align}
    H(t) &= \erf\paren{\sqrt{\frac{R(t)}{2}}} - u(t) . \notag 
\end{align}
Let us use chain rule to compute the derivative of $H$: 
\begin{align}
    H'(t) &= \frac{2}{\sqrt{\pi}} e^{-R(t)/2} \cdot \frac{R'(t)}{2 \sqrt{2 R(t)}} - u'(t)
    = \frac{R'(t)}{\sqrt{2\pi R(t)}} e^{-R(t)/2} - (1 - u(t)^2) , \notag 
\end{align}
where the last equality is by \Cref{eqn:u'}. 
By \Cref{eqn:expR}, 
\begin{align}
    e^{-R(t)/2} &= (1 - u(t)^2) \sqrt{t/u(t)} . \notag 
\end{align}
Therefore, 
\begin{align}
    H'(t) &= \frac{R'(t)}{\sqrt{2\pi R(t)}} (1 - u(t)^2) \sqrt{\frac{t}{u(t)}} - (1 - u(t)^2) \notag \\
    &= (1 - u(t)^2) \paren{ \frac{R'(t)}{\sqrt{2\pi R(t)}} \sqrt{\frac{t}{u(t)}} - 1 }
    = (1 - u(t)^2) \paren{\frac{M(t)}{\sqrt{2\pi}} - 1} , \notag 
\end{align}
by the definition of $M$ in \Cref{eqn:M}. 
Since $ u(t)\in(0,1) $, the sign of $H'$ is the same as that of $ M(t) - \sqrt{2\pi} $. 

It is easy to verify that 
\begin{align}
&&
    \lim_{t\downarrow0} M(t) &= 2 \sqrt{5/3} , & 
    \lim_{t\to\infty} M(t) &= 2 . & 
& \notag 
\end{align}
Since $ 2 < \sqrt{2\pi} < 2\sqrt{5/3} $, by \Cref{lem:M} and continuity of $M$, there exists a unique $ \tau>0 $ such that $ M(\tau) = \sqrt{2\pi} $. 
Therefore, $ H' $ is positive on $ (0,\tau) $ and negative on $ (\tau,\infty) $, that is, $H$ increases on $ (0,\tau) $ and decreases on $ (\tau,\infty) $. 

To conclude nonnegativity of $H$, it remains to study the limits of $H$ at $0$ and $\infty$ which can be easily computed: 
\begin{align}
&&
    \lim_{t\downarrow0} H(t) &= 0 , & 
    \lim_{t\to\infty} H(t) &= 0 . & 
& \notag 
\end{align}
This completes the proof of \Cref{lem:H}. 
\end{proof}

\section{Proof of \texorpdfstring{\Cref{thm:highd}}{}}
\label{app:highd}

Let $ u_1, \cdots, u_{d+1} \in \bbR^d $ be such that 
\begin{align}
&&
    \normtwo{u_i}^2 &= d , & 
    \inprod{u_i}{u_j} &= -1 , & 
& \label{eqn:frame1} 
\end{align}
for all $ i\ne j $. 
Such vectors can be obtained by taking the vertices of a regular simplex in $ \bbR^{d+1} $, rigidly embedding them into $ \bbR^d $, and suitably centering and rescaling them. 
The condition \Cref{eqn:frame1} is equivalent to 
\begin{align}
&&
    \sum_{i = 1}^{d+1} u_i &= 0_d , & 
    \frac{1}{d+1} \sum_{i = 1}^{d+1} u_i u_i^\top &= I_d ; &  
& \label{eqn:frame2}
\end{align}
see e.g.\ \cite[Section 2.3]{Fickus_Mixon}. 

Let 
\begin{align}
    X &\sim \frac{1}{d+1} \sum_{j = 1}^{d+1} \delta_{u_j / \sigma_d} , \label{eqn:Xd} 
\end{align}
where $ \sigma_d $ is defined in \Cref{eqn:sigmad}. 
To verify sub-Gaussianity of $ X $, we need the following technical lemma. 

\begin{lemma}
\label{lem:K}
Let $ n \in\bbZ_{\ge2} $ and $s\ge0$. 
Define 
\begin{align}
    K &\coloneqq \brace{ y\in\bbR^n : \inprod{y}{1_n} = 0 , \, \normtwo{y}^2 = n } . \label{eqn:K} 
\end{align}
Then 
\begin{align}
    \max_{y\in K} \frac{1}{n} \sum_{i = 1}^n e^{s y_i} &= \frac{1}{n} e^{s \sqrt{n-1}} + \paren{ 1 - \frac{1}{n} } e^{-s / \sqrt{n-1}} . \label{eqn:Kineq} 
\end{align}
% Moreover, if $ s>0 $, then the maximizers are precisely the permutations of 
% \begin{align}
%     \matrix{ \sqrt{n-1} & -\frac{1}{\sqrt{n-1}} & \cdots & -\frac{1}{\sqrt{n-1}} }^\top \in \bbR^n . \notag 
% \end{align}
\end{lemma}

\begin{proof}
Since $K$ is compact and $ y \mapsto \frac{1}{n} \sum_{i = 1}^n e^{s y_i} $ is continuous, maximizers exist. 
The case of $ s=0 $ is trivial since both sides of \Cref{eqn:Kineq} are equal to $1$. 
Assume $ s>0 $ in the rest of the proof. 
Using the method of Lagrange multipliers, we have that any maximizer $y$ must satisfy 
\begin{align}
    s e^{s y_i} &= \alpha + 2 \beta y_i , \qquad i\in[n] \label{eqn:lag}
\end{align}
for some Lagrange multipliers $ \alpha, \beta $ corresponding to the first and second constraints in $K$, respectively. 
Since $ y_0 \mapsto s e^{s y_0} - \alpha - 2 \beta y_0 $ has second derivative $ s^3 e^{s y_0} > 0 $ and is therefore strictly convex, \Cref{eqn:lag} can have at most two solutions, i.e., each $ y_i $ can take at most two values. 
Moreover, $ y_1, \cdots, y_n $ cannot be all equal, otherwise the constraint $ \inprod{y}{1_n} = 0 $ forces $ y = 0_n $, violating the constraint $ \normtwo{y}^2 = n $. 
Since reordering the coordinates of $y$ does not alter the objective value, without loss of generality, we write a maximizer $y$ as 
\begin{align}
    y &= [\underbrace{a, \cdots, a}_{k\textnormal{ times}}, \underbrace{b, \cdots, b}_{n-k\textnormal{ times}}]^\top \notag 
\end{align}
for some $ 1\le k\le n-1 $ and $ a>b $. 
The constraints in $K$ become
\begin{align}
&&
    ka + (n-k)b &= 0 , & 
    ka^2 + (n-k)b^2 &= n , & 
& \notag 
\end{align}
whose unique solution (satisfying $a>b$) is 
\begin{align}
&&
    a &= \sqrt{\frac{n-k}{k}} , & 
    b &= -\sqrt{\frac{k}{n-k}} . & 
& \notag 
\end{align}
Hence the maximal value of the objective function takes the form of 
\begin{align}
    \frac{k}{n} \exp\paren{ s\sqrt{\frac{n-k}{k}} } + \frac{n-k}{n} \exp\paren{ -s\sqrt{\frac{k}{n-k}} } , \label{eqn:obj}
\end{align}
for some $ 1\le k\le n-1 $. 
We will prove that this is further maximized by $k=1$. 
To this end, let us rewrite \Cref{eqn:obj} in terms of $a$: 
\begin{align}
    F_s(a) &\coloneqq \frac{e^{sa} + a^2 e^{-s/a}}{1 + a^2} . \notag 
\end{align}
It suffices to show that for any $s\ge0$, $ F_s $ is nondecreasing on $ [0,\infty) $. 
Differentiating $ F_s $, we get 
\begin{align}
    F_s'(a) &= \frac{a}{(1+a^2)^2} \brace{ e^{sa} \brack{ s(1/a+a) - 2 } + e^{-s/a} \brack{ s(1/a+a) + 2 } } . \notag 
\end{align}
To determine the sign of $ F_s' $, we only need to determine the sign of the term in braces. 
Under the change of variable $ r \coloneqq s (1/a + a) $, the term in braces becomes
\begin{align}
    e^{sa} (r - 2) + e^{-s/a} (r + 2)
    = e^{-s/a} \brack{ e^r (r-2) + r + 2 } , \label{eqn:er} 
\end{align}
where the equality follows by noting that $ e^{sa} = e^{-s/a} e^r $. 
Now it is easy to verify that $ g(r) \coloneqq e^r (r-2) + r + 2 $ on the RHS of \Cref{eqn:er} satisfies 
\begin{align}
&& 
    g(0) &= 0 ; &
    g'(0) &= 0 ; &
    g''(r) &\ge 0 , \textnormal{ for } r\ge0 . & 
& \notag 
\end{align}
This implies that $ g(r) \ge 0 $ for all $ r\ge0 $, which in turn implies that $ F_s'(a) \ge 0 $ for all $a\ge0$, i.e., $ F_s $ is minimized at $ 0 $. 
Therefore, \Cref{eqn:obj} is maximized at $ k=1 $. 
This completes the proof. 
\end{proof}

Equipped with \Cref{lem:K}, we verify the $ 1 $-sub-Gaussianity of $X$ in the next lemma. 

\begin{lemma}
Let $ d\in\bbZ_{\ge2} $. 
Consider $ X $ defined in \Cref{eqn:Xd}. 
Then $ X \in \cG(d) $. 
\end{lemma}

\begin{proof}
Fix any $ v\in\bbS^{d-1} $. 
Consider the vector $ y\in\bbR^{d+1} $ with $ y_i = \inprod{u_i}{v} $ for every $ i\in[d+1] $. 
By \Cref{eqn:frame2}, $ \sum_{i = 1}^{d+1} y_i = 0 $, and
\begin{align}
    \sum_{i = 1}^{d+1} y_i^2 &= \sum_{i = 1}^{d+1} \inprod{u_i}{v}^2
    = v^\top \paren{ \sum_{i = 1}^{d+1} u_i u_i^\top } v
    = (d+1) \normtwo{v}^2 = d+1 . \notag 
\end{align}
Therefore, $ y \in K $ for $K$ defined in \Cref{eqn:K} with $n=d+1$. 
Applying \Cref{lem:K}, we have that for all $s\ge0$, 
\begin{align}
    \frac{1}{d+1} \sum_{i = 1}^{d+1} e^{s \inprod{u_i}{v}} &\le \frac{1}{d+1} e^{s\sqrt{d}} + \frac{d}{d+1} e^{-s/\sqrt{d}} . \notag 
\end{align}
For $ s<0 $, applying \Cref{lem:K} to $ \abs{s} = -s > 0 $ and $ -y \in K $ yields
\begin{align}
    \frac{1}{d+1} \sum_{i = 1}^{d+1} e^{s \inprod{u_i}{v}} &\le \frac{1}{d+1} e^{\abs{s}\sqrt{d}} + \frac{d}{d+1} e^{-\abs{s}/\sqrt{d}} . \label{eqn:MGF}
\end{align}
So \Cref{eqn:MGF} holds for any $ s\in\bbR $. 

To further upper bound the RHS of \Cref{eqn:MGF}, we note that it is the moment generating function at $\abs{s}$ of a two-point distribution: 
\begin{align}
    Y &\sim \frac{1}{d+1} \delta_{\sqrt{d}} + \frac{d}{d+1} \delta_{-1/\sqrt{d}} . \notag 
\end{align}
It is easy to check that 
\begin{align}
&&
    Y &\eqqlaw \frac{d+1}{\sqrt{d}} \paren{ B - \frac{1}{d+1} } , & 
    \textnormal{where } B &\sim \bern\paren{ \frac{1}{d+1} } . & 
& \label{eqn:YB} 
\end{align}
By the Kearns--Saul inequality (stated as \cite[Lemma 1]{Kearns_Saul} and proved in \cite[Theorem 3.2]{Berend_Kontorovich}), for $ B \sim \bern(p) $, it holds that  
\begin{align}
    \expt{e^{t (B - p)}} &\le \exp\paren{ \frac{1 - 2p}{4 \log\paren{\frac{1-p}{p}}} \cdot t^2 } , \notag 
\end{align}
for any $ t\in\bbR $. 
In view of \Cref{eqn:YB}, applying this to $ B $ with $ p = 1/(d+1) $ and $ t = \abs{s} (d+1) / \sqrt{d} $, we obtain 
\begin{align}
    \expt{ e^{\abs{s} Y} } &\le \exp\paren{ \frac{1 - 2/(d+1)}{4 \log(d)} \cdot \frac{(d+1)^2}{d} s^2 }
    = \exp\paren{ \frac{\sigma_d^2 s^2}{2} } , \label{eqn:KS} 
\end{align}
where the last step follows by recalling the definition of $ \sigma_d $ in \Cref{eqn:sigmad}. 
Finally, combining \Cref{eqn:MGF,eqn:KS}, we have that for every $ \lambda\in\bbR $ and $ v\in\bbS^{d-1} $, 
\begin{align}
    \expt{e^{\lambda\inprod{X}{v}}}
    &= \frac{1}{d+1} \sum_{i = 1}^{d+1} e^{(\lambda / \sigma_d) \cdot \inprod{u_i}{v}}
    \le \expt{e^{\abs{\lambda / \sigma_d} \cdot Y}}
    \le e^{\lambda^2/2} , \notag 
\end{align}
which completes the proof. 
\end{proof}

Define $ f\colon\bbR^d \to \bbR $ as 
\begin{align}
    f(x) &\coloneqq \max_{i\in[d+1]} \inprod{u_i}{x} . \label{eqn:fd} 
\end{align}
Since $f$ is the pointwise maximum of linear functions, $f$ is convex. 
We compute
\begin{align}
    \expt{ f(X) }
    &= \frac{1}{d+1} \sum_{i = 1}^{d+1} f(u_i / \sigma_d)
    = \frac{1}{d+1} \sum_{i = 1}^{d+1} \max_{j\in[d+1]} \inprod{u_j}{u_i/\sigma_d} \notag \\
    &= \frac{1}{d+1} \sum_{i = 1}^{d+1} \max\brace{ d, -1, \cdots, -1 } / \sigma_d 
    = d/\sigma_d , \label{eqn:fXd} 
\end{align}
where the second line is by \Cref{eqn:frame1}. 

Next, let $ G \sim \cN(0_d,I_d) $. 
We claim that 
\begin{align}
    \paren{ \inprod{G}{u_1} , \cdots , \inprod{G}{u_{d+1}} }
    &\eqqlaw \sqrt{d+1} \, \paren{ Z_1 - Z_0 , \cdots , Z_{d+1} - Z_0 } , \label{eqn:law}
\end{align}
where $ Z_1, \cdots, Z_{d+1} \iid \cN(0,1) $ and $ Z_0 \coloneqq \frac{1}{d+1} \sum_{i = 1}^{d+1} Z_i $. 
To see this, note that both sides of \Cref{eqn:law} are centered Gaussian vectors since they are linear transformations of $G$ and $ (Z_i)_{i=1}^{d+1} $, respectively. 
Then it remains to check their covariance matrices. 
A straightforward calculation shows that they have the same covariance given by $ (d+1) I_{d+1} - 1_{d+1} 1_{d+1}^\top $, thereby justifying \Cref{eqn:law}. 
Using this, for any $c>0$, we compute 
\begin{align}
    \expt{f(c\,G)}
    &= c \expt{ \max_{i\in[d+1]} \inprod{u_i}{G} }
    = c \, \sqrt{d+1} \expt{ \max_{i\in[d+1]} Z_i - Z_0 }
    = c \, \sqrt{d+1} \, m_{d+1} , \label{eqn:fGd} 
\end{align}
since $ Z_0 $ has mean $0$ by definition. 

The constant $ \cmgf(d) $ must satisfy 
\begin{align}
    \expt{ f(X) } &\le \expt{ f(\cmgf(d) \cdot G) }
    = \cmgf(d) \expt{ f(G) } , \notag 
\end{align}
which, by \Cref{eqn:fXd,eqn:fGd}, is equivalent to 
\begin{align}
    \cmgf(d) &\ge \frac{d}{\sqrt{d+1} \, \sigma_d \, m_{d+1}} . \notag 
\end{align}
This completes the proof of \Cref{thm:highd}. 

\section{Proof of \texorpdfstring{\Cref{rk:2d}}{}}
\label{app:sharp2}

% To complement \Cref{thm:2d}, 
We show that in dimension $2$, the sub-Gaussian comparison constant $ \frac{8}{9} \sqrt{\pi\log(2)} $ given by \Cref{thm:highd} is in fact sharp for the specific construction of $X$ in \Cref{eqn:Xd}. 

\begin{lemma}
\label{lem:X}
Consider $X$ defined in \Cref{eqn:Xd} for $d=2$. 
Let $ G \sim \cN(0_2,I_2) $. 
Then $ X \lecx \frac{8}{9} \sqrt{\pi\log(2)} \, G $. 
\end{lemma}

\begin{proof}
For $d=2$, one can explicitly realize the construction of $X$ in \Cref{eqn:Xd} by identifying $ u_i = \sqrt{2} \, v_{i-1} $ ($ i\in\{1,2,3\} $) where 
\begin{align}
&&
    v_0 &\coloneqq \matrix{1 \\ 0} , & 
    v_1 &\coloneqq \matrix{-1/2 \\ \sqrt{3}/2} , & 
    v_2 &\coloneqq \matrix{-1/2 \\ -\sqrt{3}/2} . & 
& \notag 
\end{align}
are the vertices of an equilateral triangle centered around the origin, and can be alternatively written in polar coordinates as:
\begin{align}
&&
    v_j &= \matrix{\cos(\theta_j) \\ \sin(\theta_j)} , & 
    \textnormal{where } \theta_j &\coloneqq \frac{2\pi j}{3} . & 
& \notag 
\end{align}
Since $ \sigma_2 = \sqrt{\frac{3}{4\log(2)}} $ by definition \Cref{eqn:sigmad}, $X$ can be written as 
\begin{align}
&&
    X &\sim \frac{1}{3} \sum_{j = 0}^2 \delta_{\tau v_j} , & 
    \textnormal{where } \tau &\coloneqq \sqrt{\frac{8\log(2)}{3}} . & 
& \notag 
\end{align}
Denote $ Z \coloneqq \frac{8}{9} \sqrt{\pi\log(2)} \, G $. 
To prove convex domination, let us construct a coupling between $ X $ and $Z$. 
Partition $ \bbR^2 $ into three cones of equal angles: 
\begin{align}
    C_i &\coloneqq \brace{ x\in\bbR^2 : \inprod{v_i}{x} \ge \max_{0\le j\le 2} \inprod{v_j}{x} } 
    = \brace{ r \matrix{\cos(\theta) \\ \sin(\theta)} : r\ge0 , \, \theta \in \mleft[ \frac{(2i-1)\pi}{3}, \frac{(2i+1)\pi}{3} \mright) } , \notag 
\end{align}
for $ i \in \{0,1,2\} $. 
Define 
\begin{align}
    Y &\coloneqq \sum_{j = 0}^2 \tau v_j \indicator{G \in C_j} . \notag 
\end{align}
It is easy to see that $ Y \eqqlaw X $ since $ \prob{ Y = \tau v_j } = \prob{ G \in C_j } = 1/3 $ for every $ j\in\{0,1,2\} $. 
Moreover, we claim that 
\begin{align}
    \expt{ Z \mid Y } = Y . \label{eqn:EZY}
\end{align}
To see this, we note that 
\begin{align}
&& 
    \mleft. G \mid \brace{ G\in C_j } \mright. &\eqqlaw R \matrix{\cos(\Theta) \\ \sin(\Theta)} , & 
    \textnormal{where } (R, \Theta) \sim \sqrt{\chi_2^2} \ot \unif\paren{ \mleft[ \frac{(2j-1)\pi}{3}, \frac{(2j+1)\pi}{3} \mright) } , & 
& \notag 
\end{align}
which allows us to compute 
\begin{align}
    \expt{Z \mid Y = \tau v_j}
    &= \frac{8}{9} \sqrt{\pi\log(2)} \cdot \expt{G \mid G \in C_j}
    = \frac{8}{9} \sqrt{\pi\log(2)} \cdot \sqrt{\frac{\pi}{2}} \frac{1}{2\pi/3} \int_{(2j-1)\pi/3}^{(2j+1)\pi/3} \matrix{\cos(\theta) \\ \sin(\theta)} \diff \theta \notag \\
    &= \sqrt{\frac{8\log(2)}{3}} \matrix{\cos(2\pi j/3) \\ \sin(2\pi j/3)}
    = \tau v_j , \notag 
\end{align}
thereby justifying \Cref{eqn:EZY}. 
Using this in conjunction with the conditional Jensen's inequality, we have that for every convex function $ f \colon \bbR^2 \to \bbR $, $ f(Y) = f(\expt{Z \mid Y}) \le \expt{f(Z) \mid Y} $. 
Further averaging over $Y$, we conclude $ \expt{f(Y)} \le \expt{f(Z)} $. 
\end{proof}

\section{Proof of \texorpdfstring{\Cref{rk:highd}}{}}
\label{app:rk}

We prove that $ \cmgf(d) $ is a nondecreasing function. 

\begin{lemma}
\label{lem:mono}
For every $ d\in\bbZ_{\ge1} $, $ \cmgf(d+1) \ge \cmgf(d) $. 
\end{lemma}

\begin{proof}
This result is easily seen by lifting an arbitrary pair of $ X \in \cG(d) $ and $ f\colon\bbR^{d} \to \bbR $ convex to $ \bbR^{d+1} $. 
We briefly spell out the details below. 
Let $ X\in\cG(d) $ and $ Y = [X^\top , 0]^\top \in \bbR^{d+1} $. 
Let $ \lambda\in\bbR $ and $ u = [v^\top , w]^\top \in \bbR^{d+1} $ where $ (v,w)\in\bbR^d\times\bbR $ satisfies $ \normtwo{v}^2 + w^2 = 1 $. 
If $ v = 0_d $, then $ \expt{ e^{\lambda \inprod{Y}{u}} } = 1 \le e^{\lambda^2/2} $. 
Otherwise, 
\begin{align}
    \expt{e^{\lambda\inprod{Y}{u}}} &= \expt{e^{\lambda\inprod{X}{v}}} = \expt{e^{\lambda \normtwo{v} \cdot \inprod{X}{v/\normtwo{v}}}} \le e^{\lambda^2 \normtwo{v}^2 / 2} \le e^{\lambda^2/2} , \notag 
\end{align}
since $ \normtwo{v}^2\le1 $. 
Therefore, $ Y\in\cG(d+1) $. 

Now let $ f\colon\bbR^d \to \bbR $ be convex and define $ g\colon\bbR^{d+1}\to\bbR $ as $ g(v,w) = f(v) $ for any $ (v,w)\in\bbR^d\times\bbR $. 
Then $ g $ is also convex. 
The constant $ \cmgf(d+1) $ must satisfy $ \expt{g(Y)} \le \expt{g(\cmgf(d+1) \cdot (G, H))} $ where $ (G,H) \sim \cN(0_d,I_d) \ot \cN(0,1) $. 
This is equivalent to $ \expt{f(X)} \le \expt{f(\cmgf(d+1) \cdot G)} $. 
By minimality of $ \cmgf(d) $, we have $ \cmgf(d+1) \ge \cmgf(d) $, as desired. 
\end{proof}

\end{document}